\documentclass[12pt]{article}
\usepackage[centertags]{amsmath}
\usepackage{amsfonts}
\usepackage{amssymb}
\usepackage{amsthm}
\usepackage{newlfont}
\usepackage{graphicx}

\newfont{\bb}{msbm10 at 12pt}
\def\r{\hbox{\bb R}}

\def\e{\hbox{\bf E}}
\def\t{\hbox{\bf T}}
\def\n{\hbox{\bf N}}
\def\b{\hbox{\bf B}}

\newtheorem{theorem}{Theorem}[section]
\newtheorem{definition}[theorem]{Definition}

\newtheorem{corollary}[theorem]{Corollary}

\begin{document}

\title{Timelike $B_2$-slant helices in Minkowski space $\e_1^4$}
\author{ Ahmad T. Ali \& Rafael L\'opez\footnote{Partially
supported by MEC-FEDER
 grant no. MTM2007-61775.}}
\date{}

\maketitle
\begin{abstract} We consider a unit speed timelike curve $\alpha$ in Minkowski 4-space $\e_1^4$ and denote the Frenet frame of $\alpha$ by $\{\t,\n,\b_1,\b_2\}$. We say that $\alpha$ is a generalized helix if  one of the unit vector fields of the Frenet frame has constant scalar product with a fixed direction $U$ of $\e_1^4$.  In this work we study those helices where the function $\langle\b_2,U\rangle$ is constant and we give different characterizations of such curves.
\end{abstract}



\section{Introduction and statement of results}
A helix in Euclidean 3-space $\e^3$ is a curve   where the tangent lines make a constant angle with a fixed direction. A helix curve is characterized by the fact that the ratio $\tau/\kappa$ is constant along the curve, where $\tau$ and $\kappa$ denote the torsion and the curvature, respectively. Helices are well known curves in classical differential geometry of space curves \cite{mp} and we refer to the reader for recent
works on this type of curves \cite{gl,sc}. Recently, Izumiya and Takeuchi have introduced the concept of slant helix by saying that the normal lines make a constant angle with a fixed direction \cite{it}. They characterize a slant helix iff the function
\begin{equation}\label{slant}
\dfrac{\kappa^2}{(\kappa^2+\tau^2)^{3/2}}\Big(\dfrac{\tau}{\kappa}\Big)'
\end{equation}
is constant. The article \cite{it} motivated generalizations in a twofold sense: first, by considering arbitrary dimension of Euclidean space \cite{ky,okkk}; second, by considering analogous problems in other ambient spaces, for example, in  Minkowski space $\e_1^n$ \cite{ba,fgl,ko,ps,sy}.

In this work we consider  the generalization of the concept of helix in Minkowski 4-space,  when the helix is a timelike curve. We denote by $\e_1^4$ the  Minkowski 4-space, that is,  $\e_1^4$ is the real vector space $\r^4$ endowed  with the standard Lorentzian metric
$$\langle,\rangle=-dx_1^2+dx_2^2+dx_3^2+dx_4^2,$$
where $(x_1,x_2,x_3,x_4)$ is a rectangular coordinate system of $\r^4$. An arbitrary vector $v\in\e_1^4$ is said
spacelike (resp. timelike, lightlike) if $\langle v,v\rangle>0$ or $v=0$ (resp. $\langle v,v\rangle<0$, $\langle v,v\rangle =0$ and $v\not=0$).
Let $\alpha:I\subset\r\rightarrow \e_1^4$ be a  (differentiable) curve  with  $\alpha'(t)\not=0$, where
$\alpha'(t)=d\alpha/dt(t)$. The curve $\alpha$ is said timelike if all its velocity vectors $\alpha'(t)$ are timelike.
Then it is possible to re-parametrize $\alpha$ by a new parameter $s$, in such way that $\langle\alpha'(s),\alpha'(s)\rangle=-1$, for any $s\in I$.  We say then that $\alpha$ is a unit speed timelike curve.

Consider $\alpha=\alpha(s)$ a unit speed timelike curve in $\e_1^4$.
Let $\{\t(s),\n(s),\b_1(s),\b_2(s)\}$ be the moving frame along $\alpha$, where $\t,\n,\b_1$ and $\b_2$ denote the tangent, the principal normal, the first binormal and second binormal vector fields, respectively. Here $\t(s)$, $\n(s)$, $\b_1(s)$ and $\b_2(s)$ are mutually orthogonal vectors satisfying
$$\langle\t,\t\rangle=-1,\langle\n,\n\rangle=\langle\b_1,\b_1\rangle=\langle\b_2,\b_2\rangle=1.$$
Then the Frenet equations for $\alpha$ are given by
\begin{equation}\label{equi1}
 \left[
   \begin{array}{c}
     \t' \\
     \n' \\
     \b_1' \\
     \b_2'\\
   \end{array}
 \right]=\left[
           \begin{array}{cccc}
             0 & \kappa_1 & 0 & 0 \\
             \kappa_1 & 0 & \kappa_2 & 0 \\
             0 & -\kappa_2 & 0 & \kappa_3 \\
             0 & 0 & -\kappa_3 & 0 \\
           \end{array}
         \right]\left[
   \begin{array}{c}
     \t \\
     \n \\
     \b_1 \\
     \b_2\\
   \end{array}
 \right],
 \end{equation}

Recall the functions $\kappa_1(s)$, $\kappa_2(s)$ and $\kappa_3(s)$ are called respectively, the first, the second and the third curvatures of  $\alpha$. If $\kappa_3(s)=0$ for any $s\in I$, then $\b_2(s)$ is a constant vector $B$ and the curve $\alpha$ lies in a three-dimensional affine subspace orthogonal to $B$, which is isometric to the Minkowski 3-space $\e_1^3$.

We will assume throughout this work that all the three curvatures satisfy $\kappa_i(s)\not=0$ for any $s\in I$, $1\leq i\leq 3$.

\begin{definition} \label{df-1} A unit speed timelike curve $\alpha:I\rightarrow\e_1^4$ is said to be a generalized (timelike) helix if there exists a constant vector field $U$ different from zero and a  vector field $X\in\{\t,\n,\b_1,\b_2\}$ such that the function
$$s\longmapsto\langle X(s),U\rangle,\ \ s\in I$$
 is constant.
\end{definition}
In this work we are interested by generalized timelike helices in $\e_1^4$ where the function $\langle\b_2,U\rangle$ is constant. Motivated by the concept of slant helix in $\e^4$ \cite{okkk}, we give the following

\begin{definition} A unit speed timelike curve $\alpha$ is called a $B_2$-slant helix if there exists a constant vector field $U$ such that the function $\langle \b_2(s),U\rangle$ is constant.
\end{definition}
Our main result in this work is the following characterization of $B_2$-slant helices in the spirit of the one given in equation (\ref{slant})  for  a slant helix in $\e^3$:
\begin{quote} \emph{A unit speed timelike curve in $\e_1^4$  is a   $B_2$-slant helix if and only if the function
$$\dfrac{1}{\kappa_1^2}\Big(\dfrac{\kappa_3}{\kappa_2}\Big)^{\prime\,2}
-\Big(\dfrac{\kappa_3}{\kappa_2}\Big)^2$$
is constant.}
\end{quote}
When $\alpha$ is a lightlike curve, similar computations are been given by Erdogan and Yilmaz in \cite{ey}.

\section{Basic equations of timelike helices}

Let  $\alpha$ be a unit speed timelike curve in $\e_1^4$ and let $U$ be a unit constant vector field in $\e_1^4$. For each $s\in I$, the vector $U$ is expressed as linear combination of the orthonormal basis $\{\t(s),\n(s),\b_1(s),\b_2(s)\}$. Consider the differentiable functions $a_i$, $1\leq i\leq 4$,
\begin{equation}\label{u1}
U=a_1(s)\t(s)+a_2(s) \n(s)+a_3(s) \b_1(s)+a_4(s)\b_2(s),\ \ s\in I,
\end{equation}
that is,
$$a_1=-\langle\t,U\rangle, \ a_2=\langle\n,U\rangle,\ a_3=\langle\b_1,U\rangle,\ a_4=\langle\b_2,U\rangle.$$
Because the vector field $U$ is constant, a differentiation in (\ref{u1}) together (\ref{equi1}) gives the following ordinary differential equation system
\begin{equation}\label{u2}
\left.\begin{array}{ll}
a_1'+\kappa_1 a_2&=0\\
a_2'+\kappa_1 a_1-\kappa_2 a_3&=0\\
a_3'+\kappa_2 a_2-\kappa_3 a_4&=0\\
a_4'+\kappa_3 a_3 &=0
\end{array}\right\}
\end{equation}
In the case that $U$ is spacelike (resp. timelike), we will assume that $\langle U,U\rangle=1$ (resp. $-1$). This means that the constant $M$ defined by
\begin{equation}\label{equi2}
M:=\langle U,U\rangle=-a_1^2+a_2^2+a_3^2+a_4^2
\end{equation}
is $1$, $-1$ or $0$ depending if $U$ is spacelike, timelike or lightlike, respectively.

We now suppose that $\alpha$ is a generalized helix. This means that there exists $i$, $1\leq i\leq 4$, such that the function $a_i=a_i(s)$ is constant. Thus in the system (\ref{u2}) we have four differential equations and three derivatives of functions.

The first case that appears is that the function $a_1$ is constant, that is, the function $\langle \t(s),U\rangle$ is constant. If $U$ is  timelike, that is, the tangent lines of $\alpha$ make a constant (hyperbolic) angle with a fixed timelike direction, the curve $\alpha$ is called a timelike cylindrical helix \cite{ko}. Then it is known that  $\alpha$ is timelike cylindrical helix iff the function
$$\dfrac{1}{\kappa_3^2}\Big(\dfrac{\kappa_1}{\kappa_2}\Big)^{\prime\,2}
+\Big(\dfrac{\kappa_1}{\kappa_2}\Big)^2$$
is constant \cite{ko}.

However the hypothesis that $U$ is timelike can be dropped and we can assume that $U$ has any causal character, as for example,  spacelike or lightlike. We explain this situation. In Euclidean space one speaks  on the angle that makes a fixed direction with the tangent lines (cylindrical helices) or the normal lines (slant helices).  In Minkowski space, one can only speak about the angle between two vectors $\{u,v\}$ if both are spacelike (Euclidean angle) or both are timelike and are in the same timecone (hyperboilc angle). See \cite[page 144]{on}.  This is the reason to avoid any reference about  'angles' in  Definition \ref{df-1}.

Suppose now that the function $\langle\t(s),U\rangle$ is constant, independent on the causal character of $U$.
From the expression of $U$ in (\ref{u1}), we know that $a_1'=0$ and by using (\ref{u2}), we obtain $a_2=0$ and
$$a_3=\frac{\kappa_1}{\kappa_2}a_1,\ a_3'=\kappa_3 a_4,\ a_4'+\kappa_3 a_3=0.$$
Consider the change of variable $t(s)=\int_0^s\kappa_3(x) dx$. Then $\dfrac{d t}{ds}(s)=\kappa_3(s)$ and the last two above equations write as $a_3''(t)+a_3(t)=a_4''(t)+a_4(t)=0$. Then one obtains that there exist constants $A$ and $B$ such that
$$a_3(s)=A\cos\int_0^s\kappa_3(s)ds+B\sin\int_0^s\kappa_3(s)ds$$
$$a_4(s)=-A\sin\int_0^s\kappa_3(s)ds+B\cos\int_0^s\kappa_3(s)ds.$$
Since $a_3^2+a_4^2=\langle U,U\rangle+a_1^2$ is constant, and
$$a_4=\frac{1}{\kappa_3} a_3'=\frac{1}{\kappa_3}\Big(\frac{\kappa_1}{\kappa_2}\Big)'a_1,$$
it follows that
$$\frac{1}{\kappa_3^2}\Big(\frac{\kappa_1}{\kappa_2}\Big)'^{2}+\Big(\frac{\kappa_1}{\kappa_2}\Big)^2=\text{constant}.$$
Then one can prove the following

\begin{theorem} Let $\alpha$ be a unit speed timelike curve in $\e_1^4$. Then the function $\langle\t(s),U\rangle$ is constant for a fixed constant vector field $U$ if and only if the the function
$$\dfrac{1}{\kappa_3^2}\Big(\dfrac{\kappa_1}{\kappa_2}\Big)^{\prime\,2}
+\Big(\frac{\kappa_1}{\kappa_2}\Big)^2$$
is constant.
\end{theorem}

When $U$ is a timelike constant vector field, we re-discover the result given in \cite{ko}.

\section{Timelike $B_2$-slant helices}

Let $\alpha$ be  a  $B_2$-slant helix, that is, a unit speed timelike curve in $\e_1^4$ such that the function $\langle \b_2(s),U\rangle$, $s\in I$, is constant for a fixed constant vector field $U$ . We point out that $U$ can be of any causal character. In the particular case that $U$ is spacelike, and since $\b_2$ is too, we can say that a $B_2$-slant helix is a timelike curve whose second binormal lines make a constant angle with a fixed (spacelike) direction.

Using the system (\ref{u1}), the fact that $\alpha$ is a $B_2$-slant helix means that   the function $a_4$  is constant. Then (\ref{u2})  gives  $a_3=0$ and  (\ref{u1}) writes as
\begin{equation}\label{equi3}
U=a_1(s)\t(s)+a_2(s)\n(s)+a_4\b_2(s),\ a_4\in\r
\end{equation}
where
\begin{equation}\label{equi4}
a_2=\frac{\kappa_3}{\kappa_2}a_4=-\frac{1}{\kappa_1}a_1',\hspace*{1cm}\ a_2'+\kappa_1 a_1=0.
\end{equation}
We remark that $a_4\not=0$: on the contrary, and from (\ref{u2}), we conclude $a_i=0$, $1\leq i\leq 4$, that is, $U=0$: contradiction.

It follows from (\ref{equi4}) that  the function $a_1$ satisfies the following second order  differential equation:
$$\frac{1}{\kappa_1}\dfrac{d}{ds}\Big(\frac{1}{\kappa_1} a_1'\Big)-a_1=0.$$
If we change variables in the above equation as $\dfrac{1}{\kappa_1}\dfrac{d}{ds}=\dfrac{d}{dt}$, that is, $t=\int_0^s\kappa_1(s)ds$, then we get
$$\dfrac{d^2a_1}{dt^2}-a_1=0.$$
The general solution of this equation is
\begin{equation}\label{equi5}
a_1(s)=A\cosh\int_0^s\kappa_1(s)ds+B\sinh\int_0^s\kappa_1(s)ds,
\end{equation}
where $A$ and $B$ are arbitrary constants.
From (\ref{equi4}) and (\ref{equi5}) we have
\begin{equation}\label{equi6}
a_2(s)=-A\sinh\int_0^s\kappa_1(s)ds-B\cosh\int_0^s\kappa_1(s)ds.
\end{equation}
The above expressions of $a_1$ and $a_2$ give
\begin{equation}\label{equi7}
\begin{array}{ll}
A&=-\Big[\dfrac{\kappa_3}{\kappa_2}\sinh\int_0^s\kappa_1(s)ds
+\dfrac{1}{\kappa_1}\Big(\dfrac{\kappa_3}{\kappa_2}\Big)^{\prime}\cosh\int_0^s\kappa_1(s)ds\Big]a_4,\\
B&=-\Big[\dfrac{1}{\kappa_1}\Big(\dfrac{\kappa_3}{\kappa_2}\Big)^{\prime}\sinh\int_0^s\kappa_1(s)ds
+\dfrac{\kappa_3}{\kappa_2}\cosh\int_0^s\kappa_1(s)ds\Big]a_4.
\end{array}
\end{equation}
From  (\ref{equi7}),
$$A^2-B^2=\Big[\dfrac{1}{\kappa_1^2}\Big(\dfrac{\kappa_3}{\kappa_2}\Big)^{\prime\,2}
-\dfrac{\kappa_3^2}{\kappa_2^2}\Big]a_4^2.$$
Therefore
\begin{equation}\label{equi8}
\dfrac{1}{\kappa_1^2}\Big(\dfrac{\kappa_3}{\kappa_2}\Big)^{\prime\,2}
-\dfrac{\kappa_3^2}{\kappa_2^2}=\text{constant}:=m.
\end{equation}

Conversely, if the condition (\ref{equi8}) is satisfied for a  timelike curve, then we can always find a constant  vector  field $U$ such that the function $\langle\b_2(s),U\rangle$ is constant: it is sufficient if we define
$$U=\Big[-\dfrac{1}{\kappa_1}\Big(\dfrac{\kappa_3}{\kappa_2}\Big)'\t+
\dfrac{\kappa_3}{\kappa_2}\n+\b_2\Big].$$
By taking account of the differentiation of (\ref{equi8}) and the Frenet equations (\ref{equi1}), we have that $\dfrac{dU}{ds}=0$  and this means that $U$ is a constant vector.
On the other hand, $\langle\b_2(s),U\rangle=1$. The above computations can be summarized as follows:

\begin{theorem}\label{th-31} Let $\alpha$ be a unit speed timelike curve in $\e_1^4$. Then $\alpha$ is a   $B_2$-slant helix if and only if the function
$$\dfrac{1}{\kappa_1^2}\Big(\dfrac{\kappa_3}{\kappa_2}\Big)^{\prime\,2}
-\Big(\dfrac{\kappa_3}{\kappa_2}\Big)^2$$
is constant.
\end{theorem}

From (\ref{equi2}), (\ref{equi5}) and  (\ref{equi6}) we get
$$A^2-B^2=a_4^2-M= a_4^2\ m.$$
Thus, the sign of the constant $m$ agrees with the one $A^2-B^2$. So, if $U$ is timelike or lightlike, $m$ is positive. If $U$ is spacelike, then the sign of $m$ depends on $a_4^2-1$. For example, $m=0$ iff $a_4^2=1$. With similar computations as above, we have

\begin{corollary} Let $\alpha$ be a unit speed timelike curve in $\e_1^4$ and let $U$ be a unit spacelike constant vector field.
Then $\langle\b_2(s),U\rangle^2=1$  for any $s\in I$ if and only if there exists a constant $A$ such that
$$\dfrac{\kappa_3}{\kappa_2}(s)=A \exp{\Big(\int_0^s\kappa_1(t) dt\Big)}.$$
\end{corollary}

As a consequence of Theorem \ref{th-31}, we obtain other characterization of $B_2$-slant helices. The first one is the following

\begin{corollary} Let $\alpha$ be a unit speed timelike curve in $\e_1^4$. Then $\alpha$  is a  $B_2$-slant helix if and only if there exists real numbers $C$ and $D$ such that
\begin{equation}\label{equi9}
\dfrac{\kappa_3}{\kappa_2}(s)=C\sinh\int_0^s\kappa_1(s)ds+D\cosh\int_0^s\kappa_1(s)ds,
\end{equation}
\end{corollary}

\begin{proof} Assume that $\alpha$ is a  $B_2$-slant helix. From (\ref{equi4}) and (\ref{equi6}), the choice $C=-A/a_4$ and $D=-B/a_4$ yields  (\ref{equi9}).

We now suppose that (\ref{equi9}) is satisfied. A straightforward computation gives
$$\dfrac{1}{\kappa_1^2}\Big(\dfrac{\kappa_3}{\kappa_2}\Big)^{\prime\,2}
-\Big(\dfrac{\kappa_3}{\kappa_2}\Big)^2=C^2-D^2.$$
We now use Theorem \ref{th-31}.
\end{proof}

We end this section with a new characterization for  $B_2$-slant helices.  Let now assume that $\alpha$ is a  $B_2$-slant helix in $\e_1^4$. By differentiation (\ref{equi8}) with respect to $s$ we get
\begin{equation}\label{equi10}
\dfrac{1}{\kappa_1}\Big(\dfrac{\kappa_3}{\kappa_2}\Big)'\Big[\dfrac{1}{\kappa_1}
\Big(\dfrac{\kappa_3}{\kappa_2}\Big)'\Big]'
-\Big(\dfrac{\kappa_3}{\kappa_2}\Big)\Big(\dfrac{\kappa_3}{\kappa_2}\Big)'=0,
\end{equation}
and hence
$$\dfrac{1}{\kappa_1}\Big(\dfrac{\kappa_3}{\kappa_2}\Big)'=\dfrac{\Big(\dfrac{\kappa_3}{\kappa_2}\Big)
\Big(\dfrac{\kappa_3}{\kappa_2}\Big)'}{\Big[\dfrac{1}{\kappa_1}
\Big(\dfrac{\kappa_3}{\kappa_2}\Big)'\Big]'},$$
If we define a function $f(s)$ as
$$f(s)=\dfrac{\Big(\dfrac{\kappa_3}{\kappa_2}\Big)
\Big(\dfrac{\kappa_3}{\kappa_2}\Big)'}{\Big[\dfrac{1}{\kappa_1}
\Big(\dfrac{\kappa_3}{\kappa_2}\Big)'\Big]'},$$
then
\begin{equation}\label{equi11}
f(s)\kappa_1(s)=\Big(\dfrac{\kappa_3}{\kappa_2}\Big)'.
\end{equation}
By using (\ref{equi10}) and (\ref{equi11}), we have
$$f'(s)=\dfrac{\kappa_1\kappa_3}{\kappa_2}.$$
Conversely, consider the function $f(s)=\dfrac{1}{\kappa_1}\Big(\dfrac{\kappa_3}{\kappa_2}\Big)'$ and assume that $f'(s)=\dfrac{\kappa_1\kappa_3}{\kappa_2}$.  We compute
\begin{equation}\label{equi12}
\frac{d}{ds}\Bigg[\dfrac{1}{\kappa_1^2}\Big(\dfrac{\kappa_3}{\kappa_2}\Big)^{\prime\,2}
-\dfrac{\kappa_3^2}{\kappa_2^2}\Bigg]= \dfrac{d}{ds}\Bigg[f(s)^2-\dfrac{f'(s)^2}{\kappa_1^2}\Bigg]:=\varphi(s).
\end{equation}
As $f(s)f'(s)=\Big(\dfrac{\kappa_3}{\kappa_2}\Big)\Big(\dfrac{\kappa_3}{\kappa_2}\Big)'$ and
$f''(s)=\kappa_1'\Big(\dfrac{\kappa_3}{\kappa_2}\Big)+\kappa_1\Big(\dfrac{\kappa_3}{\kappa_2}\Big)'$
we obtain
$$f'(s)f''(s)=\kappa_1\kappa_1'\Big(\dfrac{\kappa_3}{\kappa_2}\Big)^2+\kappa_1^2\Big(\dfrac{\kappa_3}{\kappa_2}\Big)
\Big(\dfrac{\kappa_3}{\kappa_2}\Big)'.$$
As consequence of above computations
$$\varphi(s)=2\Bigg(f(s)f'(s)-\dfrac{f'(s)f''(s)}{\kappa_1^2}+\frac{\kappa_1' f'(s)^2}{\kappa_1^3}\Bigg)=0,$$
that is, the function $\dfrac{1}{\kappa_1^2}\Big(\dfrac{\kappa_3}{\kappa_2}\Big)^{\prime\,2}
-\Big(\dfrac{\kappa_3}{\kappa_2}\Big)^2$ is constant. Therefore we have proved the following

\begin{theorem} Let $\alpha$ be a unit speed timelike curve in $\e_1^4$. Then $\alpha$  is a  $B_2$-slant helix if and only if  the function $f(s)=\dfrac{1}{\kappa_1}\Big(\dfrac{\kappa_3}{\kappa_2}\Big)'$ satisfies
$f'(s)=\dfrac{\kappa_1\kappa_3}{\kappa_2}$.
\end{theorem}

\vspace*{1cm}
\emph{Complete address:}

 Ahmad T. Ali\\Mathematics Department\\
 Faculty of Science, Al-Azhar University\\
 Nasr City, 11448, Cairo, Egypt\\
email: atali71@yahoo.com\\
\vspace*{.5cm}\\
Rafael L\'opez\\
Departamento de Geometr\'{\i}a y Topolog\'{\i}a\\
Universidad de Granada\\
18071 Granada, Spain\\
email: rcamino@ugr.es

\end{document}